\documentclass[11pt,reqno]{amsart}

\usepackage{lipsum}
\usepackage{graphicx}
\usepackage{array}
\usepackage{amssymb}
\usepackage{hyperref}
\usepackage{amsthm}
\usepackage{amsfonts}
\usepackage{amsmath}
\usepackage{bm}
\usepackage{mathrsfs}
\usepackage[all]{xy}
\usepackage{color}
\usepackage{subfigure}
\usepackage{tikz, tikz-cd}
\usepackage{enumerate}
\usepackage{anysize}
\usepackage{amscd}
\usepackage[letterpaper]{geometry}
\usepackage{multirow}
\usepackage{array}
\usepackage{booktabs}

\newcommand{\p}{\partial}

\newcommand{\PP}{\mathbb{P}}
\newcommand{\CC}{\mathbb{C}}
\newcommand{\HH}{\mathbb{H}}
\newcommand{\RR}{\mathbb{R}}

\newcommand{\ZZ}{\mathbb{Z}}
\newcommand{\NN}{\mathbb{N}}

\renewcommand{\i}{\sqrt{-1}}

\newcommand{\op}{\overline{\partial}}

\renewcommand{\leq}{\leqslant}
\renewcommand{\geq}{\geqslant}

\renewcommand{\epsilon}{\varepsilon}

\newcommand{\dC}{\mathbb{C}}
\newcommand{\dN}{\mathbb{N}}
\newcommand{\dR}{\mathbb{R}}

\newcommand{\bp}{\overline\partial}

\newcommand{\fM}{\mathfrak{M}}

\DeclareMathOperator{\ALH}{ALH}

\DeclareMathOperator{\Image}{Image}

\DeclareMathOperator{\Nil}{Nil}

\DeclareMathOperator{\I}{I}

\DeclareMathOperator{\Ima}{Im}
\DeclareMathOperator{\Rea}{Re}

\newtheorem{theorem}{Theorem}[section]
\newtheorem{proposition}[theorem]{Proposition}

\newtheorem{lemma}[theorem]{Lemma}

\newtheorem{corollary}[theorem]{Corollary}

\theoremstyle{remark}

\theoremstyle{definition}

\newtheorem{definition}[theorem]{Definition}
\newtheorem{remark}[theorem]{Remark}

\numberwithin{equation}{section}
\numberwithin{figure}{section}
\numberwithin{table}{section}

\begin{document}

\title{Asymptotically Calabi metrics and weak Fano manifolds}

\begin{abstract}
We show that any asymptotically Calabi manifold which is Calabi-Yau can be compactified complex analytically to a weak Fano manifold. Furthermore, the Calabi-Yau structure arises from a generalized Tian-Yau construction on the compactification, and we prove a strong uniqueness theorem. 
We also give an application of this result to the surface case.
\end{abstract}

\author{Hans-Joachim Hein}
\address{Mathematisches Institut, WWU M\"unster, 48149 M\"unster, Germany} 
\email{hhein@uni-muenster.de}

 \author{Song Sun}
\address{Department of Mathematics, University of California, Berkeley, CA 94720} 
\email{sosun@berkeley.edu}

\author{Jeff Viaclovsky}
\address{Department of Mathematics, University of California, Irvine, CA 92697} 
\email{jviaclov@uci.edu}

\author{Ruobing Zhang}
\address{Department of Mathematics, Princeton University, Princeton, NJ 08544}
\email{ruobingz@princeton.edu}

\maketitle

\setcounter{tocdepth}{1}

\section{Introduction}
\thispagestyle{empty}

This paper is concerned with Tian-Yau construction \cite{TianYau} of complete Ricci-flat K\"ahler metrics on the complement of a smooth anti-canonical divisor in a smooth Fano manifold. We begin by describing their geometry at infinity.

\subsection{Asymptotically Calabi metrics}
We begin by defining the asymptotic models for the metrics constructed in \cite{TianYau}.  Let $D$ be an $(n-1)$-dimensional compact K\"ahler manifold with trivial canonical bundle and let $L\rightarrow D$ be an ample line bundle. Define
\begin{align}
\deg(L) = \int_D c_1(L)^{n-1},
\end{align}
and fix a nowhere vanishing holomorphic $(n-1)$-form $\Omega_D$ on $D$ satisfying
\begin{align}
\label{lalala}
\frac{1}{2}\int_{D} (\sqrt{-1})^{(n-1)^2}\Omega_D\wedge\overline{\Omega}_D=(2\pi c_1(L))^{n-1}.
\end{align} 
Using Yau's resolution of the Calabi conjecture \cite{Yau}, there exists a unique Ricci-flat K\"ahler metric $\omega_D\in 2\pi c_1(L)$ with 
\begin{align}
\omega_D^{n-1}=\frac{1}{2}(\sqrt{-1})^{(n-1)^2} \Omega_D\wedge\overline{\Omega}_D. 
\end{align}
There exists a unique hermitian metric $h$ on $L$ whose curvature form is $-\sqrt{-1}\omega_D$, up to scaling. Fixing a choice of $h$, the {\emph{Calabi model space}} is the subset $\mathcal C$ of $L$ consisting of all elements $\xi$ with $0<|\xi|_h < 1$. We next define
a nowhere vanishing holomorphic volume form $\Omega_\mathcal C$ and a Ricci-flat K\"ahler metric $\omega_\mathcal C$ which is incomplete as $|\xi|_h \to 1$ and complete as $|\xi|_h \to 0$. Let $p: \mathcal C\rightarrow D$ denote the bundle projection and  $Z$ be the holomorphic vector field generating the natural $\dC^*$-action on the fibers of $p$.
The holomorphic volume form $\Omega_{\mathcal C}$ is uniquely determined by the equation 
 \begin{align}
Z \lrcorner\ \Omega_{\mathcal C}=p^*\Omega_D,
\end{align}
and the metric  $\omega_{\mathcal{C}}$ is given by the \emph{Calabi ansatz}
 \begin{align}\label{calabiansatz}\omega_{\mathcal C}=\frac{n}{n+1}  \sqrt{-1}\p\bp (-{\log |\xi|_h^2})^{\frac{n+1}{n}},
\end{align}
which satisfies the complex Monge-Amp\`ere equation
 \begin{equation*}\omega_{\mathcal C}^n=\frac{1}{2}(\sqrt{-1})^{n^2} \Omega_{\mathcal C}\wedge\overline\Omega_{\mathcal C},\end{equation*}
hence is Ricci-flat. The function $z=(-{\log |\xi|_h^2})^{1/n}$ is the $\omega_{\mathcal{C}}$-moment map for the natural $S^1$-action on $L$. It is easily verified that the $\omega_{\mathcal{C}}$-distance function $r$ to a fixed point in $\mathcal C$ satisfies
\begin{align}
\label{e:distz}
r^{-1} z^{\frac{n+1}{2}} = C + o(1), 
\end{align}
as $z \to \infty$. 

\begin{definition}\label{d:AC}
 A structure $(X,I,\omega, \Omega)$ is asymptotically Calabi if
there exists $\underline{\delta} > 0$, a Calabi model space $\mathcal{C}$, and a diffeomorphism $\Phi: \mathcal C\setminus K'\rightarrow X\setminus K$, where $K \subset X$ is compact and $K' = \{|\xi|_h \geq \frac{1}{2}\}$, such that the following hold uniformly as $z\to+\infty$:
\begin{align}
\label{AC1}
|\nabla_{g_{\mathcal C}}^k(\Phi^*I_{X}-I_{\mathcal C})|_{g_\mathcal C} &=O(e^{-\underline{\delta} z^{n/2}}) \\
\label{AC2}
|\nabla_{g_{\mathcal{C}}}^k(\Phi^*\Omega_X-\Omega_{\mathcal C})|_{g_{\mathcal C}}&=O(e^{-\underline{\delta} z^{n/2}}) \\
\label{AC3}
|\nabla_{g_\mathcal C}^k(\Phi^*\omega-\omega_{\mathcal C})|_{g_{\mathcal C}}&=O(e^{-\underline{\delta} z^{n/2}}).
\end{align} 
for all $k \in \mathbb{N}_0$. 
\end{definition}

Our main theorem is the following:
\begin{theorem} 
\label{t:main1}
Any asymptotically Calabi structure $(X,I, \omega, \Omega)$ which is Calabi-Yau can be compactified complex analytically to a weak Fano manifold $\overline{X}$.  Furthermore, the Calabi-Yau structure arises from a generalized Tian-Yau construction on $\overline{X}$ and $\omega$ is the unique Calabi-Yau metric with respect to $(I, \Omega)$ satisfying \eqref{AC3} and representing $[\omega] \in H^2(X)$. 
\end{theorem}
The above theorem solves a particular case of Yau's conjecture in \cite{YauICM}.
We give a precise description of this generalized Tian-Yau construction in Section~\ref{s:dP}.
The main step in the proof is to show that for a certain choice of a parallel complex structure, $I$, the underlying complex manifold of an Calabi-Yau asymptotically Calabi metric  can be compactified to a weak Fano manifold. This involves producing holomorphic functions with controlled growth at infinity, which is typically done using weighted Fredholm theory. This strategy runs into trouble due to the fact that, in the Tian-Yau construction, the decay rate of the metric is much slower than the decay rate of the complex structure. To overcome this difficulty, our strategy is based on the $L^2$-estimates in several complex variables pioneered by H\"ormander \cite{Hormander}.  
The proof of Theorem \ref{t:main1} can be found in Section \ref{s:dP}.
We also have the following corollary. Recall the index of $\overline{X}$ is the largest integer $k$ such that $K_X^{-1} = H^{k}$ for some line bundle $H$. 
\begin{corollary}
\label{c:main1} Let $(X^n,I,\omega, \Omega)$ be
an asymptotically Calabi manifold which is Calabi-Yau. Then $\pi_1(X)$ is a cyclic group with order the index of $\overline{X}$.  Furthermore, there exists a constant $C(n)$, depending only upon $n$, such that $\deg(L) \leq C(n)$. 
\end{corollary}
This degree bound is remarkable because in Definition~\ref{d:AC}, $\deg(L)$ could a priori be any integer, but if $\mathcal{C}$ occurs at infinity for a Calabi-Yau asymptotically Calabi metric, then $\deg(L)$ must be bounded. The proof of this degree bound uses Theorem~\ref{t:main1}
and deep results in birational geometry.

\subsection{The surface case}
A gravitational instanton is by definition a complete noncompact hyperk\"ahler $4$-manifold  $(X,g, \bm{\omega})$ with square-integrable curvature. By the results of the recent paper \cite{SZ21}, a gravitational instanton is always asymptotic to a model end. Accordingly, gravitational instantons can be classified into 6 families: ALE, ALF, ALG, ALH, ALG$^*$, ALH$^*$. There has been extensive work, much of it quite recent, on classifying the 6 families completely \cite{CCI, CCII, CCIII, CVII, KronII, MinTor}. This results in this paper are relevant to the ALH$^*$ family. This has the unique intriguing feature that its members have fractional asymptotic volume growth; indeed, the volume growth exponent of an ALH$^*$ model end is $\frac{4}{3}$. Gravitational instantons of type $\ALH^*$ also appear as singularity models in polarized degenerations of K3 surfaces \cite{HSVZ, SZ}. Their precise definition can be found in Section~\ref{s:ALHstar}. 

There are two known mathematical constructions of ALH$^*$ gravitational instantons. Both of them come with a preferred choice of complex structure, and are based on solving a complex Monge-Amp\`ere equation on a quasiprojective surface with trivial canonical bundle. First, we have the Tian-Yau construction \cite{TianYau},  which involves the complement of a smooth anticanonical divisor in a del Pezzo surface. Second, we have the construction of \cite{Hein}, which involves the complement of a singular fiber of Kodaira type $\I_b$ in a rational elliptic surface. Our next theorem relates the complex structures involved in these two constructions. 

\begin{theorem}
\label{t:main2}  
Let $(X,g,\bm{\omega})$ be an $\ALH^*$ gravitational instanton.

\begin{enumerate}[(i)]
\item Letting $I$ denote the complex structure corresonding to $\omega_1$, then $(X,I)$ is biholomorphic to weak del Pezzo surface $\overline{X}$ minus a smooth anticanonical elliptic curve. Furthermore, the hyperk\"ahler structure arises from a generalized Tian-Yau construction on this compactification and is the unique Calabi-Yau metric with respect to $(I, \Omega = \omega_2 + \sqrt{-1} \omega_3)$ satisfying \eqref{AC3} and representing $[\omega_1] \in H^2(X)$.

\item Letting $J$ denote the complex structure corresponding to $I_2$, then
$(X,J)$ compactifies to a rational elliptic surface $S$ with a global section by adding $F$, a Kodaira type $\I_b$ fiber of multiplicity~$1$. The $2$-form $\Omega=\omega_2+ \i\omega_3$ is a rational 2-form on $S$ with a simple pole along~$F$.
\end{enumerate}
\end{theorem}

The proof of Theorem~\ref{t:main2} will be given in Section~\ref{s:ALHstar} and is our 
originally intended proof of some claims made in \cite[Remark 2.4]{HSVZ}. The main ingredients in the proof are Theorem \ref{t:main1}, the decay estimates of \cite[Section~3]{HSVZ}, and the analysis of harmonic functions on asymptotically Calabi spaces of \cite[Section~4]{HSVZ}. We also note that in the meantime Collins-Jacob-Lin have proved, using an entirely different method, that a Tian-Yau space can be compactified to a rational elliptic surface; see \cite[Theorem 1.3]{CJL}.

\subsection{Acknowledgements.} The authors would like to thank Gao Chen for valuable discussions on gravitational instantons. HJH was supported by NSF grant DMS-1745517 and by the DFG under Germany's Excellence Strategy EXC 2044-390685587 ``Mathematics M\"unster:~Dynamics-Geometry-Structure'' as well as by the CRC 1442 ``Geometry:~Deformations and Rigidity'' of the DFG. SS was supported by the Simons Collaboration on Special Holonomy in Geometry, Analysis and Physics ({\#}488633), and NSF grant DMS-2004261.  JV was partially supported by NSF Grants DMS-1811096 and DMS-2105478.
RZ was partially supported by NSF Grants DMS-1906265 and DMS-2304818.

\section{Compactification to weak Fano manifold}
\label{s:dP}

In this section, we give the proof of Theorem \ref{t:main1}. 
Let $(X, I, \omega, \Omega)$ be an asympotically Calabi Calabi-Yau manifold. We identify $X\setminus K$ smoothly with a Calabi model space $\mathcal C \setminus K'$, where $K' = \{z \leq z_0\}$ for some $z_0 \gg 0$, and assume that \eqref{AC1}--\eqref{AC3} are satisfied. Let $\phi_0\equiv z^n-\delta z^{n/2}$ for some $\delta\in (0,  1/2)$ to be chosen later. 
Then 
\begin{align}
\begin{split}
dd^c_{\mathcal C}\phi_0 &= - d J_{\mathcal{C}} d \phi_0 \\
&= - d \Big( \big(nz^{n-1} - (n/2)\delta z^{n/2 - 1}\big) J_{\mathcal{C}} dz \Big)\\
&= d \Big( ( n - (n/2) \delta z^{-n/2}) \theta \Big)\\
&= (n - (n/2)\delta z^{-n/2}) d \theta + \delta (n/2)^2 z^{-n/2 - 1} dz \wedge \theta,
\end{split}
\end{align}
where $\theta \equiv - z^{n-1} J_{\mathcal{C}} dz$. Note also that
\begin{align}
\begin{split}
\omega_D = \sqrt{-1} \partial \overline{\partial} ( - \log \Vert \xi \Vert_h^2) 
= \sqrt{-1} \partial \overline{\partial} ( z^n)  = n d \theta,
\end{split}
\end{align}
and
\begin{align}
\begin{split}
\omega_{\mathcal{C}} &= \frac{n}{n+1} \frac{1}{2} d d^c_{\mathcal{C}} z^{n+1} = - \frac{n}{2(n+1)}   d J_{\mathcal{C}} d z^{n+1}\\
& = - \frac{n}{2} d (z^{n} J dz) = \frac{n}{2} d (z \theta) = \frac{n}{2} ( z d \theta + dz \wedge \theta).
\end{split}
\end{align}
Then for $z_1 \gg z_0$, 
\begin{align}
\begin{split}
d d_{\mathcal{C}} \phi_0 & = z^{-n/2-1} \Big(   ( nz^{n/2+1} - (n/2)\delta z )d \theta +
 \delta (n/2)^2 dz \wedge \theta \Big)\\
& =  z^{-n/2-1} \Big(   ( nz^{n/2+1} - (n/2)\delta z - (n/2)^2 \delta z  )d \theta  
 +  \delta (n/2)^2 (z d \theta + dz \wedge \theta) \Big)\\
& \geq z^{-n/2-1} (n/2) \delta \omega_{\mathcal{C}},
\end{split}
\end{align}
since $n d \theta = \omega_D$ is positive definite on $D$.  
Let $A=z_1^n-\delta z_1^{n/2}$, and choose a smooth increasing and convex function $u: \dR\rightarrow \dR$ such that $u(t)=2A/3$ for $t\leq A/2$ and $u(t)=t$ for $t\geq A$. Denote $\phi_1 \equiv u \circ \phi_0$. Then 
\begin{align}
d d^c_{\mathcal{C}} \phi_1 &= - d J_{\mathcal{C}} d( u \circ \phi_0) = -d J_{\mathcal{C}} u'(\phi_0) d \phi_0 = - u''(\phi_0) d \phi_0 \wedge J_{\mathcal{C}} d\phi_0 
+ u'(\phi_0) d d^c_{\mathcal{C}} \phi_0.
\end{align}
Since the form $-  d \phi_0 \wedge J_{\mathcal{C}} d\phi_0$ is positive semi-definite, we see that $dd^c_{\mathcal C} \phi_1\geq 0$ for all $z\geq z_0$ and
$dd^c_{\mathcal C}\phi_1\geq C\delta z^{-n/2-1}\omega_{\mathcal C}$  for $z>z_1$.

Notice that $\phi_1$ can be naturally viewed as a smooth function on $X$, and it satisfies $dd_I^c\phi_1\geq \Phi(z)\omega$ for a nonnegative function $\Phi(z)$ with $\Phi(z)\geq C\delta z^{-n/2-1}$ when $z$ is large. 
In particular, we know that $(X, I)$ is 1-convex. So, by \cite[Section 2]{Grauert}, there is a Remmert reduction $\pi: X\rightarrow \widetilde X$, where $\widetilde X$ is Stein and $\text{Sing}(\widetilde X)$ is a finite set contained in the region $\{z\leq z_1\}$. Then $\widetilde X$ admits an exhaustion function $\psi_{\widetilde X}$ which is smooth on $\widetilde X\setminus \text{Sing}(\widetilde X)$ and satisfies $dd^c_I\psi_{\widetilde X}>0$.  Denote $\psi_X \equiv \pi^*\psi_{\widetilde X}$. Then $dd^c_I\psi_X>0$ on $X\setminus E$, where $E\equiv\pi^{-1}(\text{Sing}(\widetilde X))$. 
Choose a cutoff function $\chi$ on $X$ supported in $\{z\leq z_1+1\}$ with $\chi\equiv 1$ on $\{z\leq z_1\}$. Then let 
$\phi \equiv \epsilon\chi\psi_X+\phi_1$ for a fixed $0<\epsilon\ll1$. The above calculation shows that $\phi$ also satisfies  $dd^c_I\phi\geq \Phi(z)\omega$, where $\Phi(z)\geq 0$ on $X$, $\Phi(z)>0$ on $X\setminus E$, and  $\Phi(z)\geq C \delta z^{-n/2-1}$ outside a compact set.  

Any holomorphic section $s\in H^0(D, L^k)$ gives rise to a holomorphic function $f_s$ on $L\setminus {\bf 0}_L$ by defining  
\begin{align}
\label{fsdef}
f_s(\xi) = s(\pi(\xi))/ \xi^{\otimes^k},
\end{align} 
where $\xi \in L\setminus {\bf 0}_L$, and $\pi : L \rightarrow D$ is the bundle projection. In particular,  $f_s$ restricts to a holomorphic function on $\mathcal C$. Taking the logarithm of \eqref{fsdef}, and using that $z^n = - 2 \log | \xi|_h$, we see that
$f_s = \tilde{f}_s e^{\frac{k}{2}z^n}$, where $\tilde{f}_s$ is a function on the unit circle bundle of $L$. 
\begin{lemma}
\label{l:elliptic}
 We have  $|\nabla_{g_{\mathcal{C}}}^lf_s|_{g_{\mathcal{C}}}=O(e^{\frac{k}{2}z^n + \epsilon z^{n/2}})$ for all $l\geq0$ and any $\epsilon > 0$, as $z \to \infty$. 
\end{lemma}
\begin{proof}
The estimate for $l = 0$ holds with $\epsilon = 0$, and follows from the previous remarks. Since $f_s$ is holomorphic, it is harmonic. The curvature of the Calabi metric is in particular bounded at infinity; see \cite[Lemma~4.3]{TianYau} and \cite[Remark~3.2]{HSVZ}. Given $p \in \mathcal{C}$, let $B_1(p)$ be a unit ball around $p$, and $B_1(\tilde{p})$ be a unit ball in its universal cover. Then standard elliptic estimates for harmonic functions yield that
\begin{align}
|\nabla^l f_s|(p) \leq \Vert \nabla^l f_s \Vert_{C^0(B_{1/2}(\tilde{p}))} \leq C_l \Vert f_s \Vert_{C^0(B_{1}(\tilde{p}))}.
\end{align}
Using \eqref{e:distz}, one can easily check that if $|r(q) - r(p)| < 1$, then 
\begin{align}
|z(q)^n - z(p)^n| \leq C z(p)^{(n-1)/2}.
\end{align} 
Since $(n-1)/2 < n/2$, the claim follows. 
\end{proof}
Denote by $\mathcal O(X)$ the space of $I$-holomorphic functions on $X$.

\begin{proposition}
\label{p:L2 estimate}
There is an injective linear map $\mathcal L: \bigoplus_{k=0}^\infty H^0(D, L^k)\rightarrow \mathcal O(X)$ such that for any nonzero section $s\in H^0(D, L^k)$, we have that $|\nabla_g^l(\mathcal L(s)-f_s)|_{g}=O(e^{\frac{k}{2}z^n-\frac{\delta_k}{2}z^{n/2}})$ for all $l\geq 0$ and for some $\delta_k>0$.
\end{proposition}

\begin{proof}
First fix a cutoff function $\chi$ on $\mathcal C$ which is equal to $1$ when $z\geq z_1+1$ and  vanishes when $z\leq z_1$. Then for any $s\in H^0(D, L^k)$, the function $\chi f_s$ naturally extends to a smooth function on $X$. Notice that for $z>z_0+1$  we have that
\begin{align}
|\bp_I(\chi f_s)|_{g}=|\bp_If_s|_{g}=|(\bp_I-\bp_{\mathcal C})f_s|_{g}\leq e^{\frac{k}{2}z^n-\underline{\delta}z^{n/2}}. 
\end{align}
Set $\delta\equiv\min(\frac{\underline{\delta}}{2k}, \frac{1}{2})$ in the definition of $\phi_0$ above. 
Then we have that
\begin{align}
\int_{X\setminus E} \frac{1}{\Phi(z)}|\bp_I(\chi f_s)|_g^2e^{-k\phi}\,d{\rm Vol}_g<\infty.
\end{align}
Notice that $X\setminus E\cong \widetilde X\setminus \text{Sing}(\widetilde{X})$ admits a complete K\"ahler metric (see \cite[Proposition 4.1]{Ohsawa}). Also, by assumption $K_X$ is trivial, so we can apply the standard $L^2$-estimates for the $\overline\partial$-operator on $X\setminus E$ (see for example \cite[Chapter VIII.6, Theorem 6.1]{Demailly}) to find a unique solution $u$ to the equation  $\bp_Iu=\bp_I(\chi  f_s)$ with 
\begin{align}
\int_{X\setminus E}|u|^2e^{-k\phi}\,d{\rm Vol}_g\leq\int_{X\setminus E} \frac{1}{\Phi(z)}|\bp_I(\chi f_s)|_g^2e^{-k\phi}\,d{\rm Vol}_g
\end{align} such that $u$ is $L^2$ orthogonal to ${\rm ker}(\bp_I)$. Notice that $\Delta_gu=\bp_{I,g}^*\bp_I(\chi f_s)=O(e^{(k/2)z^n-(\underline{\delta}/2)z^{n/2}})$. Similar to the proof of Lemma~\ref{l:elliptic}, it follows from local $L^2$ elliptic estimates that $|\nabla^l_gu|_g=O(e^{\frac{k}{2}z^n- \delta z^{n/2}})$ for all $l\geq0$.   Now let $\mathcal L(s)\equiv\chi f_s-u$. This function is holomorphic away from $E$, so by Hartogs's theorem (applied to $\widetilde{X}$) one can see that it is globally holomorphic on $X$. The conclusion then follows.
\end{proof}

Fix $k$ such that $L^l|_D$ is very ample for all $l\geq k$. Then we have a holomorphic embedding $F_k:L \rightarrow\mathbb P(H^0(D, L^k)^*\oplus H^0(D, L^{k+1})^*)$ defined by 
\begin{align}
(x, \xi)\in L\mapsto  ({\rm ev}_{x,k} / \xi^{\otimes k}, {\rm ev}_{x,k+1} / \xi^{\otimes(k+1)}),\end{align}
where $x\in D$, $\xi\in L_x$, and ${\rm ev}_{x, l}:H^0(D, L^l)\rightarrow L^l|_{x}$ is the evaluation map. Alternatively, we can describe $F_k$ as follows. By assumption, we have the embeddings 
\begin{align}
i_{L^k} : D \rightarrow  \mathbb P(H^0(D, L^{k})^*), \quad i_{L^{k+1}} : D \rightarrow \mathbb P(H^0(D, L^{k+1})^*).
\end{align}
We can view $i_{L^k}$ and $i_{L^{k+1}}$ as mapping into $\mathbb P(H^0(D, L^k)^*\oplus H^0(D, L^{k+1})^*)$. Then for $x \in D$, $F_k$ maps the fiber $\pi^{-1}(x) \subset L$ linearly to the line between $i_{L_k}( \pi(p))$
and $i_{L^{k+1}}(\pi(p))$, so it is clearly an embedding. Obviously,  $F_k({\bf 0}_L)$ is isomorphic to $D$ and is contained in the linear subspace $\mathbb P(H^0(D, L^{k+1})^*) \subset  \mathbb P(H^0(D, L^k)^*\oplus H^0(D, L^{k+1})^*)$.

Now we define a holomorphic map $G_k: X\rightarrow \mathbb P(H^0(D, L^k)^*\oplus H^0(D, L^{k+1})^*)$ via
\begin{align}
p\in X\mapsto (\widetilde{{\rm ev}}_{p, k}, \widetilde{{\rm ev}}_{p, k+1}),
\end{align}
where $\widetilde{{\rm ev}}_{p, k}: H^0(D, L^k)\rightarrow \mathbb C$ is given by $\widetilde{{\rm ev}}_{p, k}(s)=\mathcal L(s)(p)$. 

We denote by $\overline{X}$ the topological compactification of $X$ by adding $F_k(D)$ to the end of $G_k(X)$. This is justified by the following.
\begin{proposition}
\label{p:2-3}
There exists a compact set $K \subset X$ such that $G_k$ is a holomorphic embedding on $X\setminus K$ with $G_k (X \setminus K) \cap F_k(D) = \emptyset$.  
Furthermore there is a neighborhood of $F_k(D)$ in $\overline{X}$
which is homeomorphic to a neighborhood of $F_k({\mathbf{0}}_L)$ in $F_k(L)$.  
\end{proposition}

\begin{proof}
The key point is that we can compare $G_k$ with $F_k$ via the fixed embedding of the end of $X$ into $\mathcal C\subset L$. Given any point $q_0\in {\bf 0}_L\cong D$, we can find sections $s_0, \dots, s_{n-1}\in H^0(D, L^k)$ and $s_n\in H^0(D, L^{k+1})$ such that $s_0(q_0)\neq 0$, $s_1(q_0) = \cdots = s_{n-1}(q_0) = 0$, $ds_1(q_0), \dots, ds_{n-1}(q_0)$ are linearly independent, and $s_n(q_0)\neq0$. Then $w_k\equiv s_k/s_0$ for $1 \leq k \leq n-1$ and $w_n\equiv s_0/s_n$ form local holomorphic coordinates in a neighborhood $U$ of $q_0$ in $L$. To clarify this definition, note that $w_n$ is a local section of $L^{-1}$ on $D$, but by duality we can view such a section as a local function on the total space of $L$ which is linear on fibers. We therefore can think of $w_1, \dots, w_{n-1}$ as coordinates on the divisor, and $w_n$ as a fiber coordinate. Denote this coordinate system by $w : U \rightarrow \CC^n$. Notice that $|\xi|^2_{h_L}=|w_n|^2e^{-\varphi(w_1, \dots, w_{n-1})}$ for a smooth function $\varphi$, satisfying $\sqrt{-1} \partial \overline{\partial} \varphi = \omega_D$. The Calabi metric in the $(w_1,\dots, w_n)$ coordinates is given by
 \begin{align}
\label{Cw1w2}
 \omega_\mathcal{C}=(-\log |\xi|^2_{h_L})^{\frac{1}{n}}\omega_D+\frac{1}{n}(-\log |\xi|_{h_L}^2)^{\frac{1}{n}-1}\cdot \sqrt{-1} \cdot \Big(\frac{dw_n}{w_n}-\p\varphi\Big)\wedge \Big(\frac{d\bar w_n}{\bar w_n}-\bp\varphi\Big).
 \end{align}

We define a mapping $\Psi$ from an open subset $V \subset P(H^0(D, L^k)^*\oplus H^0(D, L^{k+1})^*)$ containing ${F}_k(q_0)$ to $\CC^n$ by
\begin{align}
[ (\alpha, \beta)] \mapsto \Big( \frac{\alpha(s_1)}{\alpha(s_0)}, \dots,  \frac{\alpha(s_{n-1})}{\alpha(s_0)}, \frac{\alpha(s_0)}{\beta(s_n)} \Big).
\end{align}
Then the restriction of $\Psi$ to $\Image(F_k)$ is a coordinate chart for $\Image(F_k)$ near ${F}_k(q_0)$, and $\Psi \circ F_k = w$ in a neighborhood of $q_0 \in {\bf{0}}_L$.
Next, let $\eta_j=\mathcal L(s_j)/\mathcal L(s_0)$ for $1 \leq j \leq n-1$, and $\eta_n=\mathcal L(s_0)/\mathcal L(s_n)$, and we denote this mapping $\eta: \Phi(U \setminus {\bf{0}}_L) \rightarrow \CC^n$ (after possibly shrinking $U$). Note that for any $p$ in the domain of $\eta$, we have
\begin{align}
\begin{split}
\Psi ( G_k(p)) &= \Psi ( [ \tilde{ev}_{p,k}, \tilde{ev}_{p,k+1}])\\
&= \Big( \frac{  \tilde{ev}_{p,k}(s_1)}{ \tilde{ev}_{p,k}(s_0)}, 
\dots,  \frac{\tilde{ev}_{p,k}(s_0)}{ \tilde{ev}_{p,k+1}(s_n)} \Big)\\
&= \Big( \frac{ \mathcal{L}(s_1)}{\mathcal{L}(s_0)} (p), \dots, 
\frac{ \mathcal{L}(s_0)}{\mathcal{L}(s_n)} (p) \Big) = \eta(p).
\end{split}
\end{align}

By Proposition \ref{p:L2 estimate}, on $U\setminus {\bf{0}}_L$ we have that $\eta \circ \Phi = w (1+\zeta)$ with $|\nabla_g^l\zeta|_g=O(e^{-\delta' z^{n/2}})$ for all $l\geq 0$ for some $\delta' > 0$. In the following, the constant $\delta'$ is allowed to change from line to line. Using \eqref{Cw1w2}, we have that
\begin{align}
\Big( \frac{\partial (\eta\circ\Phi)_\alpha}{\partial w_\beta} \Big)
=
\left(
\begin{matrix} 
   &   & &   \vdots \\
& \Large{\mathbb{I}_{n-1}} + O(e^{-\delta' z^{n/2}})  &  &  O(e^{- \delta' z^{n/2}} |w_n|^{-1}) \\
& & &   \vdots \\ 
\cdots & O(e^{- \delta' z^{n/2}} |w_n|)  & \cdots  & 1 + O(e^{-\delta' z^{n/2}}) \\
\end{matrix}
\right)
\end{align}
Thus, the Jacobian matrix $(\frac{\p (\eta \circ \Phi)_\alpha}{\p w_\beta})$ is nondegenerate for $z\gg 1$, which implies that $G_k$ is an immersion outside a compact set. 

Next, we show that $G_k$ is injective onto its image for $z\gg1$. Suppose we have two points $p_1, p_2\in X$ with $G_k(p_1)=G_k(p_2)$, and $ z(q_2)\geq z(q_1)\gg1$, where $\Phi(q_j) = p_j$ . Let $d_{FS}$ denote the Fubini-Study distance on $\mathbb P(H^0(D, L^k)^*\oplus H^0(D, L^{k+1})^*)$. We claim that there exists $\epsilon >0$ such 
that $d_{FS}(F_k(q), G_k(\Phi(q)) ) < \epsilon$ if $q \in L$ is sufficiently near $D$. 
To see this, using the formula
\begin{align}
d_{FS} (Z,W) = \arccos \sqrt{  \frac{ \langle Z, W \rangle^2}{ \Vert Z \Vert^2 \Vert W \Vert^2 } },
\end{align}
and Proposition \ref{p:L2 estimate}, we obtain that 
\begin{align}
d_{FS}( F_k(q), G_k (\Phi(q))) = \arccos \big(1 + O(e^{- \delta' z(q)^{n/2}})\big),
\end{align}
 as $z \to \infty$, since $L^{k+1}$ is very ample. Then 
\begin{align}
d_{FS}( F_k( q_1), F_k(q_2)) 
\leq d_{FS}( F_k(q_1), G_k(p_1)) + d_{FS} (G_k(p_2), F_k(q_2)) < 2 \epsilon. 
\end{align} 
This implies that $ q_1$ and $q_2$ must be contained in the same $w$ coordinate patch above, and therefore $\eta_\alpha(p_1)=\eta_\alpha(p_2)$ for $\alpha=1, \dots, n$. Denote $\tau_\alpha \equiv |w_\alpha(q_2)-w_\alpha(q_1)|$. Then $\tau_\alpha=w_\alpha(q_1)O(e^{-\delta' z(q_1)^{n/2}})$. So we know in particular that  $|w_n(q_1)|\leq C|w_n(q_2)|$.  Notice that
\begin{equation}
\label{e2.1}
|\tau_\alpha|=|w_\alpha(q_2)\zeta_\alpha(q_2)-w_\alpha(q_1)\zeta_\alpha(q_1)|\leq C e^{-\delta' z(q_1)^{n/2}}|\tau_\alpha|+|w_\alpha(q_1)|\cdot |\zeta_\alpha(q_2)-\zeta_\alpha(q_1)|	.
\end{equation}
Let $q_t$ $(t\in [0, 1])$ be the straight line connecting $q_1$ and $q_2$ in the $w$ coordinates.  Then 
\begin{equation}
\label{e2.2}
|\zeta_\alpha(q_2)-\zeta_\alpha(q_1)|\leq \sum_{\beta=1}^n\sup_{t\in [0,1]} |\frac{\p \zeta_\alpha}{\p w_\beta}(q_t)|\cdot |\tau_\beta|\leq Ce^{-\delta' z(q_1)^{n/2}}\Big(\sum_{j =1}^{n-1}|\tau_j|+\frac{1}{|w_n(q_1)|}|\tau_n| \Big).	
\end{equation}
Combining \eqref{e2.1} and \eqref{e2.2} with $\alpha=n$, we get that when $z_1\gg1$, 
\begin{equation*}
|\tau_n|\leq Ce^{-\delta' z(q_1)^{n/2}}
|w_n(q_1)|\cdot \sum_{j = 1}^{n-1}|\tau_j|.
\end{equation*}
If we now sum \eqref{e2.1} for $\alpha=1\dots n-1$ and use \eqref{e2.2}, then we obtain that $\tau_\alpha=0$ for $1 \leq \alpha \leq n-1$  and hence that $\tau_n=0$ as well.

Next, we claim that away from a sufficiently large compact subset, $G_k$ maps no point into $F_k(\bf{0}_L)$.  To see this, assume that there exists a point $p \in X$ with $G_k(p) = F_k(x)$, where $x \in \bf{0}_L$. Letting $q = \Phi^{-1}(p)$, if $z(q)$ is sufficiently large, then $q$ and $x$ must be contained in a same $w$-coordinate chart as above. Taking the $n$th component of $\Psi( G_k(p)) = \eta(p)$, yields
\begin{align}
0 = w_n(x) = \Psi_n(F_k(x)) = \eta_n(p) =  \eta_n(\Phi(q)) = w_n(q)( 1 + \zeta_n(q)).
\end{align}
But if $z(q) \gg 1$, we obtain a contradiction since $w_n(q) \neq 0$. The proposition then follows from the above. 
\end{proof}

\begin{remark}
Recall the coordinate $( \eta \circ \Phi)_n = w_n (1 + \eta_n)$, 
and since  $e^{-\delta z^{n/2}} \sim e^{-\delta'\sqrt{ -\log |w_n|}}$, the complex structure $I$ is not even H\"older continuous at the divisor.  Thus, one cannot invoke any known weak version of the Newlander-Nirenberg theorem to construct a complex-analytic compactification; see for example \cite{HillTaylor}.  This is the reason why we had to deviate from the known approaches to constructing such compactifications; compare for example the proofs of \cite[Theorem 3.1]{HHN} and \cite[Theorem 1.6]{ChiLi}. 
\end{remark}

Denote by $Z$ the image of $\overline{G}_k$. 
It follows from the Remmert-Stein theorem \cite[Chapter II.8, Theorem 8.7]{Demailly} that $Z$ is a complex analytic variety in a neighborhood of $F_k(\textbf{0}_L)\cong D$. Since $Z$ is topologically a manifold by Proposition \ref{p:2-3}, it follows that $Z$ is locally irreducible. Thus, by \cite[Chapter II.7, Corollary 7.13]{Demailly}, the normalization map $Z'\rightarrow Z$ is a homeomorphism.

Denote by $D'$ the copy of $D$ in $Z'$. Notice that, in the proof of Proposition \ref{p:2-3}, the functions  $\eta_\alpha$ for $1 \leq \alpha \leq n$ can be viewed as local holomorphic functions on a neighborhood of a point $q\in D'$ in $Z'$. Moreover, they define a local topological embedding of $Z'$ into $\mathbb C^n$. So the inverse map is holomorphic, which  implies that $Z'$ must be smooth near $D'$. 
Now we may glue $Z'$ to $X$ and obtain a smooth complex-analytic compactification of $X$, which we denote by $\overline X'$. By construction, $D'$ is a smooth divisor in $\overline{X}'$. 

Denote $\omega' \equiv\sqrt{-1}(\sum_{j=1}^{n-1}d\eta_j\wedge d\bar \eta_j+|\eta_n|^{-2}d\eta_n\wedge d\bar \eta_n)$. This is a locally defined K\"ahler metric in a punctured neighborhood of $D'$, with cylindrical behavior normal to $D'$. It is easy to check using the calculations in the proof of Proposition \ref{p:2-3} that for all $l\geq 0$, 
\begin{equation}
\label{e:estimates on coordinate change}
\sum_{j=1}^{n-1} |\nabla_{\omega'}^l(w_j)|_{\omega'}+|w_n^{-l}\nabla^l_{\omega'}(w_n)|_{\omega'}\leq C_l.
\end{equation}
This is a crucial fact for us. One can also reinterpret this as saying that the metric $\omega'$ is $C^\infty$ uniformly equivalent to the corresponding cylindrical K\"ahler metric defined using $(w_1, \dots, w_n)$. 
Note that $|\p_{w_j}z |\leq C z^{1-n}$ for $1 \leq j \leq n-1$ and $|\p_{w_n}z|\leq Cz^{1-n}|w_n|^{-1}$. In particular, we have that 
\begin{align}
\label{e:nlz}
|\nabla^l_{\omega'}z|_{\omega'}\leq C_lz^{1 -nl}.
\end{align}

Recall that $\Omega$ is the holomorphic volume form on $(X, g)$ with respect to $I$, and $\Omega_{\mathcal C}$ is the corresponding holomorphic $2$-form on the Calabi model space $\mathcal C$. Notice that $\Omega_{\mathcal C}$ is a meromorphic $2$-form on $L$ with a simple pole along the zero section ${\bf 0}_L$, and $\Omega$ can be viewed as a holomorphic $2$-form on $\overline{X}'\setminus D'$.

\begin{lemma}
$\Omega$ is a meromorphic volume form on $\overline{X}'$ with a simple pole along $D'$. In particular, $D'$ is an anticanonical divisor.
\end{lemma}

\begin{proof}
We may locally near a point $p\in D$ write  $\Omega_{\mathcal C}=f(w_1, \dots, w_n){w_n}^{-1}dw_1 \wedge \cdots \wedge dw_n$, where $f$ is a nowhere vanishing local holomorphic function. Then  $\Omega=f(w_1, \dots,  w_n){w_n}^{-1}dw_1 \wedge \cdots \wedge dw_n +\Gamma$ with $|\frac{\Gamma\wedge \overline{\Gamma}}{\omega_{\mathcal C}^n}|=O(e^{-\underline\delta z^{n/2}})$. But $\omega_{\mathcal C}^n=\frac{1}{2}( \sqrt{-1})^{n^2}\Omega_{\mathcal C}\wedge\bar\Omega_{\mathcal C}\sim -|w_n|^{-2}dw_1 \wedge d\bar w_1 \wedge \cdots \wedge dw_n \wedge d\bar w_n$ near $D$. It follows that $C^{-1}\leq |\Omega|_{\omega'}\leq C,$ so the form $w_n \Omega$ extends holomorphically to $D'$ locally.  
\end{proof}

Now we choose a finite open cover $\{V_j\}$ of $D$ such that each $V_j$ has holomorphic coordinates $w_{1, j}, \dots, w_{n-1,j}$ given by the quotients $s_{k, j}/s_{0, j}$ for some sections $s_{0, j}, \dots, s_{n-1, j}\in H^0(D, L^k)$, $1 \leq k \leq n-1$,  and $L|_{V_j}$ has a trivializing section $e_j=s_{n, j}/s_{0, j}$ for some $s_{n, j}\in H^0(D, L^{k+1})$.  Consequently we obtain an open cover $\{U_j\}$ of  a neighborhood of $\textbf{0}_L\simeq D$ in $L$, with  holomorphic coordinates $\{w_{1, j}, \dots, w_{n, j}\}$, such that a point $\xi\in L$ is represented as $\xi=w_{n, j} \cdot e_j(w_{1, j}, \dots, w_{n-1,j})$. Write $e^{-\varphi_j(w_{1,j}, \dots, w_{n-1,j} )}=|e_j(w_{1,j}, \dots, w_{n-1,j})|_{h}^2$. 

Accordingly, we get an open cover $\{U_j'\}$ of a neighborhood of $D'$ in $\overline X'$. In each $U_j'$ we then have holomorphic coordinates $\eta_{k, j}=\mathcal L(s_{k, j})/\mathcal L(s_{0, j})$ and $\eta_{n, j}=\mathcal L(s_{0, j})/\mathcal L(s_{n, j})$. From the proof of Proposition \ref{p:2-3} we know that $\eta_{k, j} = w_{k, j}$ on $U_j'\cap D'$ for $1 \leq k \leq n-1$, and the transition functions satisfy $\frac{\partial\eta_{n,j}}{\partial\eta_{n, i}}  =\frac{\partial w_{n, j}}{\partial w_{n, i}}$ on $U_i'\cap U_j'\cap D'$.  This implies that the conormal bundle $N_{D'}^{-1}$ is isomorphic to $L^{-1}$, so in particular $D'$ has ample normal bundle in $\overline{X}'$. Moreover, the local functions $\eta_{n, j}$ can be viewed as a global section $S_{D'}$ of the line bundle $[D']\simeq K_{\overline{X}'}^{-1}$ with a simple zero along $D'$. 

Now the local functions $\varphi_j$ define a hermitian metric on $N_{D'}\simeq [D']$. Fix an arbitrary extension of this to a hermitian metric   $h'$ on $[D']$ on a neighborhood of $D'$. On each chart $U_j$, this extension $h'$ has a local representation given by $e^{-\phi_j}$, with $\phi_j=\varphi_j$ along $D'$. 
We consider a smooth closed $(1,1)$-form    
\begin{align}
\omega_m \equiv \frac{n}{n+1} \sqrt{-1} \partial \overline{\partial} (-{\log\,} |S_{D'}|^2_{h'})^{\frac{n+1}{n}}.
\end{align}
It is well-defined and positive definite outside a compact subset of $\overline X'\setminus D'$. 

One can check that $\omega_m$ and $\omega'$ satisfy
\begin{align}
\label{e:op1}
 C^{-1}   z^{1-n} \omega' \leq \omega_m \leq C z \omega',
\end{align}
for some constant $C$, and for each $k \geq 0$, 
\begin{align}
\label{e:op2}
|\nabla^k_{\omega'} \omega_m|_{\omega'} = O( z^{m_k} )
\end{align}
for some $m_k \in \NN_0$, as $z \to \infty$. Therefore covariant derivative with respect to $\omega'$ and $\omega_m$ differ by at most polynomial error terms.  
\begin{proposition}
\label{l:11 form decay}
For all $l>0$ and all $\delta<\min(\delta_k, \underline{\delta})$, we have that
\begin{align}
\label{omega-m}
|\nabla_{\omega_m}^l(\omega-\omega_m)|_{\omega_m}=O(e^{-\delta z^{n/2}}).
\end{align}
Consequently, $[\omega]\in {\rm im}(H^2_c(X) \to H^2(X))$.
\end{proposition}

\begin{proof}
By assumption, $\omega=\omega_{\mathcal C}+\beta_1$, where $|\nabla^l_{\omega_{\mathcal C}}\beta_1|_{\omega_{\mathcal{C}}}=O(e^{-\underline{\delta}z^{n/2}})$ for all $l\geq 0$. 
Let $\tilde{\omega} \equiv \frac{n}{n+1} \sqrt{-1} \partial_I \overline{\partial}_I z^{n+1}$.
Writing  $\omega_{\mathcal C}=\tilde{\omega}+\beta_2$, using \eqref{e:nlz}, \eqref{e:op1}, and \eqref{e:op2}, we have that $|\nabla^l_{\tilde{\omega}}\beta_2|_{\tilde{\omega}}=O(e^{-\delta z^{n/2}})$ for all $l\geq 0$ and $\delta<\underline{\delta}$. In $U_j$,  we have that $$z^n=-{\log\,} |\xi|_{h_L}^2=-{\log\,} |w_{n, j}|^2+\varphi_j(w_{1,j}, \dots, w_{n-1,j}).$$
On the other hand,  we have that
$$|S_{D'}|^2_{h'}=-{\log\,} |\eta_{n, j}|^2+\phi_j( \eta_{1,j}, \dots, \eta_{n,j}).$$
Notice that $\eta_{n,j}=w_{n,j}(1+\beta_3)$, with $|\nabla^l_{\omega}\beta_3|_\omega =O(e^{-\delta_k z^{n/2}})$ for all $l\geq 0$, and $\phi_j(\eta_{1,j}, \dots, \eta_{n,j})=\varphi_j(w_{1,j}, \dots, w_{n-1,j})+\eta_{n,j}P(\eta_{1,j}, \dots, \eta_{n, j})$ for some smooth function $P$. Estimate \eqref{omega-m} now follows from a straightforward computation, using the key property \eqref{e:estimates on coordinate change}, \eqref{e:op1}, and \eqref{e:op2}.

By integrating from infinity as in \cite[Lemma 3.7]{HSVZ}, away from a compact set we may write
 $\omega= \omega_m +d\sigma$, where $\sigma$ is a smooth real-valued $1$-form such that $|\nabla^l_{\omega_m}\sigma|_{\omega_m} =O(e^{-\delta z^{n/2}})$ for all $l\geq 0$. Then the smooth real-valued $2$-form
\begin{align}
\omega - d \Big( \chi \Big(\sigma + \frac{n}{2(n+1)} d^c_I \left( - \log |S_{D'}|^2_{h'} \right)^{\frac{n+1}{n}}
\Big) \Big)
\end{align}
is cohomologous to $\omega$ and has compact support, where $\chi$ is a cutoff function which is $1$ in a neighborhood of infinity and $0$ on a compact set.
\end{proof}

Now we claim that $\overline X'$ is a \emph{weak Fano manifold}, i.e., $\overline{X}'$ is projective and $K_{\overline X'}^{-1}$ is big and nef. Assuming projectivity, the nef property is obvious because $K_{\overline X'}^{-1} = [D']$ for some smooth divisor $D'$ with ample normal bundle, and it is also elementary to deduce the big property from this fact (see for example \cite[Lemma 2.3]{Voisin}). Projectivity follows from a similar reasoning as in the proof of  \cite[Theorem 2.1]{ConlonHein}; see in particular \cite[p.4, proof in the smooth case]{ConlonHein}. In short, one first proves using the theory of the Remmert reduction (see the beginning of Section \ref{s:dP}) that $K_{\overline X'}^{-1}$ is semiample. This also strongly relies on the fact that $K_{\overline X'}^{-1} = [D']$, where $D'$ has ample normal bundle. It then follows that $\overline X'$ is Moishezon, which implies that Hodge theory holds on $\overline X'$. Also, $h^{0,2}(\overline X') = 0$ thanks to the Grauert-Riemenschneider vanishing theorem \cite[Satz 2.1]{GR}. Using the latter two properties, the fact that $\overline X'$ admits compactly supported K\"ahler classes by Proposition \ref{l:11 form decay}, and another vanishing theorem on the open manifold $X$ due to Grauert-Riemenschneider \cite[p.278, Korollar]{GR}, one can explicitly write down a K\"ahler form on $\overline X'$. The Kodaira embedding theorem now implies that $\overline X'$ is projective because we already know that $h^{0,2}(\overline X') = 0$. We note that a similar gluing argument of K\"ahler forms will appear again in the proof of Lemma \ref{kahlergluing} below. The argument works here even though, unlike in \cite{ConlonHein}, $D'$ is {not} Fano, so that $H^{1,1}(\overline X')$ may not surject onto $H^2(X)$. The key point is that we are restricting ourselves to the case of compactly supported K\"ahler classes on $X$.

To finish the proof of Theorem \ref{t:main1}, it remains to identify the Ricci-flat K\"ahler metric $\omega$ on $X$ with a slight generalization of the construction given by Tian-Yau \cite[Theorem 4.2]{TianYau}. Let $M$ be an $n$-dimensional projective manifold containing a smooth divisor $D$ with ample normal bundle such that $K_M^{-1} = [D]$. Then $M$ is necessarily weakly Fano. Fix a defining section $S\in H^0(M, K_M^{-1})$ of the divisor $D$ and denote $X\equiv M\setminus D$. Following the beginning of the argument from \cite{ConlonHein} sketched above, one first proves that $K_M^{-1}$ is semiample. Here this also follows from the Kawamata basepoint free theorem \cite[Theorem 6.1]{Kawa} because we already know that $M$ is projective. Let $E$ be the non-ample locus of $K_M^{-1}$. This is a union of some subvarieties of $M$ that are disjoint from $D$. We  fix a smooth hermitian metric $\tilde h$ on $K_M^{-1}$ such that its curvature form is nonnegative everywhere and is positive away from $E$. Denote $\tilde v\equiv-{\log\,} |S|^2_{\tilde h}$. 
	
We may view $S^{-1}$ as a holomorphic volume form $\Omega_X$ on $X$ with a simple pole along $D$. Let $\Omega_D$ be the  holomorphic volume form on $D$ given by the residue of $\Omega_X$ along $D$. Let $h_D$ be a hermitian metric on $K_M^{-1}|_D$ such that its curvature form is a Calabi-Yau metric $\omega_D$ on $D$. Rescaling $S$ if necessary, we may assume that
$\omega_D^{n-1}=\frac{1}{2}(\sqrt{-1})^{(n-1)^2}\Omega_D\wedge\bar\Omega_D.$
	
One can extend $h_D$ to a smooth hermitian metric $h$ on $K_M^{-1}$ such that its curvature form is positive definite in a neighborhood of $D$. Fix such an extension, denoted by $h_M$.   For any $A\in \mathbb R$, we denote $h_A \equiv h_Me^{-A}$. Then it is easy to see that outside a compact subset of $X$, $(-{\log\,} |S|_{h_A}^2)^{\frac{n+1}{n}}$  is strictly plurisubharmonic. As in the beginning of this section, by composing with a convex function we may find a global smooth function $v_A$ on $X$ with $dd^c v_A\geq 0$ such that $v_A=\frac{n}{n+1}(-{\log\,} |S|_{h_A}^2)^{\frac{n+1}{n}}$ outside a compact set and $v_A$ is constant in a neighborhood of $E$. 
	
Abusing notation, we denote by $H^2_{c, +}(X)$ the subset of ${\rm im}(H^2_c(X) \to H^2(X))$ consisting of classes $\mathfrak{k}$ such that $\int_Y \mathfrak{k}^p>0$ for any compact analytic subset $Y$ of $X$ of pure complex dimension $p > 0$. Notice that any such $Y$ must be contained in $E$.

\begin{lemma}
\label{kahlergluing}
For all $\mathfrak{k}\in H^2_{c,+}(X)$ and all numbers $A\in \mathbb R$, $\tau>0$, there exists a smooth K\"ahler form $\omega_{A,\tau}$ on $X$ such that $[\omega_{A,\tau}]=\mathfrak{k}$ and $|\nabla^l_{\omega_{A,\tau}}(\omega_{A, \tau}-\tau dd^cv_A)|_{\omega_{A,\tau}} = O(|S|^{\delta_0}_{h_M})$ for some $\delta_0>0$ and all $l \geq 0$.
\end{lemma}
	
\begin{proof} 
By hypothesis, $\mathfrak{k}$ is represented by a closed $2$-form $\beta$ on $M$ such that $\beta = 0$ away from some compact subset of $M$. In particular, $\beta$ trivially extends to a smooth closed $2$-form on $M$. By the Grauert-Riemenschneider vanishing theorem \cite[Satz 2.1]{GR}, or (because we are assuming that $M$ is projective) by the Kawamata-Viehweg vanishing theorem (see for example \cite[Theorem 2.64]{KM}), we have that $h^{0,2}(M) = 0$. Thus, there exist a smooth closed $(1,1)$-form $\eta$ and a smooth $1$-form $\gamma$ on $M$ such that $\beta = \eta + d\gamma$. Because $\beta$ has compact support in $X \subset M$, it follows that $\eta$ is $d$-exact on some open neighborhood of $D$ in $M$. In particular, $\eta|_D = dd^c f$ for some smooth function $f$ on $D$. We choose an arbitrary smooth extension of $f$ to $M$ and replace $\beta$ by $\eta - dd^c f$. Thus, we are now assuming without loss of generality that $\beta$ is a smooth closed $(1,1)$-form on $M$ such that $\beta|_D = 0$ and $[\beta|_M] = \mathfrak{k}$. It is straightforward to check that such a form $\beta$ satisfies the exponential decay estimate from the statement of the lemma with respect to any reference metric of the form $\tau dd^c v_A$ outside a sufficiently large compact subset of $X$. (If $n \geq 3$, then instead of applying the $dd^c$-lemma on $D$ it is possible to apply a stronger $dd^c$-lemma on the complement of a large compact subset of $X$ to arrange directly that $\beta$ still has compact support in $X$; see for example \cite[Lemma 4.3]{vanC}.)

Because $\mathfrak{k} \in H^2_{c,+}(X)$, by the generalized Demailly-P\u{a}un criterion of \cite[Theorem 1.1]{CT} there exists a smooth function $u_0$ on $X$ such that $\beta+dd^cu_0$ is positive in a neighborhood $U$ of $E$. Choose a cutoff function $\chi$ of compact support contained in $U$ which equals $1$ on a neighborhood of $E$, and let $\rho=\rho(t)$ be a cutoff function on $[0,\infty)$ which equals $1$ when $t\leq1$ and vanishes when $t\geq 2$. For fixed $A$ and $\tau$, we define
\begin{align}
\omega_{A, \tau} \equiv \beta+dd^c\Big(\chi u_0+C_1\Big(\rho\circ \Big(\frac{1}{C_2} \tilde{v}\Big)\Big) \cdot  \tilde{v}+\tau v_A\Big).
\end{align}
It is straightforward to verify that if we first choose $C_1$ large and then $C_2$ large (depending on $C_1)$, then $\omega_{A, \tau}$ is globally positive on $X$.
\end{proof}

One can also see that $\omega_{A, \tau}^n=e^{-f_{A, \tau}}\tau^n \frac{1}{2}(\sqrt{-1})^{{n^2}}\Omega_X\wedge\bar\Omega_X$
for some function $f_{A, \tau}$ which tends to zero at infinity.

\begin{lemma}
\label{l:fix the constant A}
Given any $\tau>0$, there is a unique $A=A(\tau)$ such that
\begin{align}
\int_X (\omega_{A,\tau}^n-\tau^n \frac{1}{2}(\sqrt{-1})^{n^2}\Omega_X\wedge\bar\Omega_X)=0.
\end{align}
\end{lemma}

\begin{proof}
It is easy to check that the integral is finite for any $A$. For a given $A$, add and subtract $\omega_{0,\tau}^n$ under the integral sign and split up the integral accordingly. For $\epsilon>0$ sufficiently small, we have 
$$\int_{|S|_h\leq\epsilon}(\omega_{A,\tau}^n-\omega_{0, \tau}^n)=\tau\int_{|S|_h=\epsilon}(d^cv_A\wedge \omega_{A,\tau}^{n-1}-d^cv_0\wedge\omega_{0,\tau}^{n-1}).$$
By computing the boundary term explicitly and letting $\varepsilon\to 0$, one sees that 
$$\int_X(\omega_{A, \tau}^n-\omega_{0, \tau}^n)= \lambda \cdot A\tau^{n}\int_D \omega_D^{n-1}$$
for some $\lambda = \lambda(n)\neq 0$. Then the desired condition becomes a linear equation for $A$.
\end{proof}

Using Lemmas \ref{kahlergluing} and \ref{l:fix the constant A} as ingredients, we now get the following existence theorem, which generalizes the classical existence result of Tian-Yau \cite[Theorem 4.2]{TianYau}.

\begin{theorem}
\label{t:TY}
Given any $\mathfrak{k}\in H^2_{c,+}(M)$ and $\tau>0$, there is a complete K\"ahler metric $\omega_\tau=\omega_{A(\tau),\tau}+dd^c\phi$ on $X$ such that $[\omega_\tau]=\mathfrak{k}$ and  $\omega_\tau^n=\tau^n\frac{1}{2}(\sqrt{-1})^{n^2}\Omega_X\wedge\bar\Omega_X,$ and such that for some $\delta_0>0$ and for all $l\geq 0$ we have that $|\nabla_{\omega_{A(\tau),\tau}}^l\phi|_{\omega_{A(\tau),\tau}}=O(e^{-\delta_0 (-{\log\,} |S|^2_{h_M})^{\frac{1}{2}}}).$ 
\end{theorem}

This follows from \cite[Theorem 1.1]{TianYau} and \cite[Proposition 2.4, Proposition 2.9(ia) and (ii)]{Hein}. The correct choice of $A = A(\tau)$ is crucial because otherwise the relevant Monge-Amp\`ere equation cannot be solved by direct methods. Once the equation has been solved for $A = A(\tau)$, it follows that a solution exists for all $A$ by adding $v_{A(\tau),\tau} - v_{A,\tau}$ to the potential, but $v_{A(\tau),\tau} - v_{A,\tau}$ is comparable to $(A(\tau)-A)\tau (-{\log}\,|S|_{h_M}^2)^{\frac{1}{n}}$ at infinity, which is not even uniformly bounded for $A \neq A(\tau)$.

\begin{remark}
\label{decay rate}
The decay rate of the complex structure  is of order $O(e^{-(\frac{1}{2}-\varepsilon)z^n})$ for all $\varepsilon>0$. This is much faster than the decay rate of the metric, which is in general only of order $O(e^{-\underline\delta z^{n/2}})$. 
\end{remark}

\begin{theorem}
\label{t:uniqueness}
Suppose we have a K\"ahler metric $\omega$ on $X$ such that $\omega^n=C\Omega_X\wedge\bar\Omega_X$ for some $C>0$ and such that there exist $\tau>0$ and $A\in \mathbb R$ such that  $|\nabla^l_{\omega}(\omega-\tau dd^c v_A)|_\omega=O(e^{-\delta (-{\log\,} |S|^2_{h})^{\frac{1}{2}}})$ for some $\delta>0$ and for all $l\geq 0$. Then $\omega=\omega_{\tau}$, the metric constructed in Theorem \ref{t:TY} in the class $\mathfrak{k} = [\omega]\in H^2_{c,+}(X).$
\end{theorem}

\begin{proof}
By rescaling $\omega$ if necessary, we may assume that $C=1$. As in the proof of Proposition \ref{l:11 form decay}, we can see that  $\omega=\tau dd^c v_A+d\sigma$ outside a compact set, where $\sigma$ is a smooth real-valued 1-form such that $|\nabla^l_{\omega_m}\sigma|_{\omega_m} =O(e^{-\delta (-{\log\,} |S|^2_{h})^{\frac{1}{2}}})$ for all $l\geq 0$. In particular, $[\omega] \in {\rm im}(H^2_c(X) \to H^2(X))$. Thus, if we let $\mathfrak{k}$ denote the de Rham cohomology class of $\omega$ in $H^2(X)$, then we have that $\mathfrak{k} \in H^2_{c,+}(X)$. Consequently, we can apply Theorem \ref{t:TY} to obtain a complete Calabi-Yau metric $\omega_{\tau}$ on $X$ with $[\omega_\tau] = \mathfrak{k}$. By construction, we have that $\omega^n=\omega_\tau^n$.
Then, as in the proof of Lemma \ref{l:fix the constant A}, we can see that $A=A(\tau)$, so we know that $|\nabla^l_{\omega_\tau}(\omega-\omega_\tau)|_{\omega_\tau}=O(e^{-\delta_1 (-{\log\,} |S|^2_{h})^{\frac{1}{2}}})$ for some $\delta_1>0$. 
 
Next,  we solve the Poisson equation $\bp^*\bp u=\bp^*\sigma^{0,1}$ with respect to $\omega_\tau$. Note that $\int_{X} (\bp^*\sigma^{0,1})\, \omega_{\tau}^n=0$ by definition, so we can apply \cite[Theorem 1.5]{Hein2} to conclude the existence of a $C^\infty$ solution $u$ which is uniformly bounded and satisfies $\int_X |du|^2_{\omega_\tau}\omega_\tau^n<\infty$. Moreover, by the integration argument in the proof of \cite[Proposition 2.9(ia)]{Hein}, after subtracting a constant from $u$ one actually has the asymptotics $|\nabla^l_{\omega_\tau}u|_{\omega_\tau}=O(e^{-\delta_2 (-{\log\,} |S|^2_{h})^{\frac{1}{2}}})$ for some $\delta_2>0$ and all $l \geq 0$. Let $\gamma \equiv \sigma^{0,1}-\bp u$. By construction, $\bp\gamma =\bp^*\gamma =0$ with respect to $\omega_\tau$. Thus, by the Bochner formula, $\Delta|\gamma|^2\geq0$. Since $\gamma$ tends to zero at infinity,  it must vanish identically. It follows that  $\omega-\omega_\tau=dd^c(\sqrt{-1}(\bar u-u))$. Finally, one can integrate by parts to conclude that $u\equiv0$.
\end{proof}
 
The second statement in Theorem~\ref{t:main1} follows from Theorem \ref{t:uniqueness} (applied to $M=(\overline X',\overline{I})$) and Proposition \ref{l:11 form decay}.

\begin{proof}[Proof of Corollary~\ref{c:main1}] By \cite[Theorem~1.1]{Takayama}, the weak Fano manifold $\overline{X} = X \cup D$ is simply connected. Since $-K_{\overline{X}}$ is big, $D$ does not consist of a pencil, so by \cite[Corollary~2.10]{Nori}, $\pi_1(X) = \pi_1( -K_{\overline{X}}^{\times})$, where $-K_{\overline{X}}^{\times} =-K_{\overline{X}} \setminus {\mathbf{0}}_{-K_{\overline{X}}}$. Letting $N$ denote the unit circle bundle of $-K_{\overline{X}}$, the homotopy sequence of the fibration $S^1 \rightarrow N \rightarrow \overline{X}$ yields 
\begin{align}
\pi_2(\overline{X}) \rightarrow \ZZ \rightarrow \pi_1(X) \rightarrow 0, 
\end{align}
since $\overline{X}$ is simply-connected, and since $- K_{\overline{X}}^{\times}$ deformation retracts onto $N$, so $\pi_1(N) \cong \pi_1( -K_{\overline{X}}^{\times}) \cong \pi_1(X)$. 
By \cite[Proposition~2]{Shimada}, the first mapping factors as the Hurewicz projection from $\pi_2(\overline{X}) \rightarrow H_2(\overline{X}, \ZZ)$ composed with the mapping $\delta: H_2(\overline{X}, \ZZ) \rightarrow \ZZ$ given by the intersection paring with $[D] \in H_{n-2}(\overline{X}, \ZZ)$, that is 
\begin{align}
\delta (\Sigma) = \int_{\Sigma} c_1([D]).
\end{align}
By Poincar\'e duality, this mapping is nontrivial, so the image is an infinite subgroup $k \ZZ \subset \ZZ$ , and consequently $\pi_1(X) = \ZZ/k\ZZ$ is cyclic, with $k$ equal to the index of $\overline{X}$. 

The bound on the degree is a consequence of the following results. It is a classical result that the degree of a weak del Pezzo surface satisfies $1 \leq d \leq 9$; see for example \cite[Chapter~8]{Dolgachev}.  A bound on the degree of a weak Fano threefold is proved in \cite[Corollary~1.3]{KMMT}. Finally, Birkar proved that in any dimension there is an a priori bound on the degree of a weak Fano manifold; see~\cite[Theorem~1.1]{Birkar}.
\end{proof}

\section{ALH$^*$ gravitational instantons}
\label{s:ALHstar}

The $3$-dimensional Heisenberg group is 
\begin{equation}
H(1,\dR) \equiv\left\{\begin{bmatrix}
1 & x &t \\
0 & 1 &  y\\
0 & 0 & 1
\end{bmatrix}: \ x,y,t\in\dR\right\}.
\end{equation}
Let $(\mathbb{T}^2, g_{c,\tau})$ be the flat $2$-torus corresponding to the lattice $c(\ZZ \oplus \ZZ \tau) \subset \CC$ with $\tau \in \mathbb{H}$ and $c > 0$, where $\mathbb{H}$ is the upper half-plane in $\CC$. Write $\tau_1 = \Rea(\tau)$ and $\tau_2 = \Ima(\tau)$. For $b \in \mathbb{Z}_+$, let $\Nil^3_{b, c, \tau}$ be the nilmanifold corresponding to the quotient of the Heisenberg group by the subgroup generated by 
\begin{align}
(x,y,t) &\mapsto (x + c, y, t + c y),\\
(x,y,t) &\mapsto (x + c \tau_1, y+ c \tau_2, t + c \tau_1 y ),\\
(x,y,t) &\mapsto (x, y, t +  \tau_2 b^{-1} c^2  ).
\end{align}

\begin{definition}
For $b \in \ZZ_+, c > 0, \tau \in \mathbb{H}$, and $R > 0$, the $\ALH^*$ model space is 
\begin{align}
\fM_{b, c, \tau}(R) \equiv (R, \infty) \times \Nil^3_{b,\tau,c}
\end{align}
together with the hyperk\"ahler Riemannian metric 
\begin{align}
g^{\fM} \equiv  V ( dx^2 + dy^2 + dz^2) + V^{-1}  \theta^2,
\end{align}
where $z$ is the coordinate on $(R, \infty)$, $V\equiv 2\pi b c^{-2} \tau_2^{-1} z $, and $\theta \equiv 2\pi b c^{-2} \tau_2^{-1}(dt - x dy)$.
\end{definition}

Choose the following orthonormal frame for $T^*\fM$:
\begin{align}
\{e_1, e_2, e_3, e_4\} = \{V^{\frac12} dx,  V^{\frac12} dy,  V^{\frac12} dz, V^{-\frac12} \theta \}.
\end{align}
Define three almost-complex structures on $T^*\fM$ by 
\begin{align}
I_{\fM}^* ( e_1) &= e_2, \quad I_{\fM}^*(e_3) = e_4, \\
\label{Jstar}
J_{\fM}^* (e_1) &= e_4, \quad J_{\fM}^* (e_2) = e_3, \\
K_{\fM}^* (e_2) &= e_4, \quad K_{\fM}^* (e_3) = e_1,
\end{align}
which are dual to almost-complex structures $I_{\fM}, J_{\fM}, K_{\fM}$ on $T\fM$, respectively.
Denote the K\"ahler forms associated to $I_{\fM}^*, J_{\fM}^*, K_{\fM}^*$
by $ \omega_1^{\fM}, \omega_2^{\fM}, \omega_3^{\fM}$, respectively. These are explicitly given by
\begin{align}
\omega_1^{\fM} & = dz \wedge \theta + V dx \wedge dy, \\
\omega_2^{\fM} &= dx \wedge \theta + V dy \wedge dz, \\
\omega_3^{\fM} &= dy \wedge \theta + V dz \wedge dx. 
\end{align}
Taken together, this data defines a hyperk\"ahler structure on $\fM$, which we denote as $(\fM, g^{\fM}, \bm{\omega}^{\fM})$. This structure can equivalently also be specified as $(\fM, g^{\fM}, I_{\fM}, J_{\fM}, K_{\fM})$. 

\begin{definition}[$\ALH^*$ gravitational instanton]
\label{d:ALHstarmetric}
A hyperk\"ahler structure $(X,g,\bm{\omega})$ on a $4$-manifold $X$ is called an $\ALH^*$ gravitational instanton with parameters $b \in \ZZ_+$, $c > 0$, and $\tau \in \mathbb{H}$, if there exist $\delta,R>0$, a compact subset $X_R\subset X$, an $\ALH^*$ model space $(\fM_{b, c, \tau}(R), g^{\fM})$, and a diffeomorphism 
\begin{equation}
\Phi:  \fM_{b, c,\tau}(R) \rightarrow X \setminus X_R
\end{equation} 
such that for all $k \in \dN_0$,
\begin{align}
\label{metricdecay}
|\nabla_{g^{\fM}}^k(\Phi^*g-g^{\fM})|_{g_{\fM}}&=O(e^{-\delta  z}), \\
\label{omegadecay}
|\nabla_{g^{\fM}}^k(\Phi^*\omega_i - \omega^{\fM}_i)|_{g_{\fM}}&=O(e^{-\delta  z}), 
\ i = 1,2,3, 
\end{align}
as $z \to \infty$.
\end{definition}

According to \cite[Proposition 3.1]{HSVZ}, any ALH$^*$ model space $\fM$ together with its complex structure $I_{\fM}$ is holomorphically isometric to a \emph{Calabi model space} up to rescaling. This means the following. We can view the flat torus $(\mathbb{T}^2,g_{c,\tau})$ as an elliptic curve $D$. Then there exists an ample line bundle $L \to D$ of degree $b$, together with a hermitian metric $h_L$ whose curvature form defines the flat metric
\begin{align}
2\pi b c^{-2}\tau_2^{-1}g_{c,\tau}
\end{align}
on $D$, such that the underlying complex manifold $(\fM_{b,c,\tau}(R), I_{\fM})$ can be identified with the open set
\begin{align}
\mathcal{C} \equiv \{\xi \in L: 0<|\xi|_{h_L}<e^{-\frac{1}{2}z_0^2}\}
\end{align}
for some $z_0 > 0$. Moreover, the K\"ahler form and the holomorphic $2$-form on $\fM$ are respectively given by $\omega_1 = \mu \omega_{\mathcal{C}}$ for some $\mu > 0$ and $\omega_2 + \sqrt{-1}\omega_3 = \nu\Omega_{\mathcal{C}}$ for some $\nu \in \mathbb{C}^*$, where 
\begin{align}
\omega_{\mathcal C}\equiv \frac{2}{3} \sqrt{-1} \partial_{\mathcal C} \overline{\partial}_{\mathcal C}(-{\log |\xi|^2_{h_L}})^{\frac{3}{2}}
\end{align}
and $\Omega_{\mathcal C}$ is a holomorphic volume form which has a simple pole along the zero section $\textbf{0}_L$ and is invariant under the natural $\mathbb C^*$-action on $L$. In addition, we have that
\begin{align}
z=(-{\log |\xi|^2_{h_L}})^{\frac12},
\end{align}
and this is the $\omega_{\mathcal{C}}$-moment map for the natural $S^1$-action on $L$. We note here that $h_L$ is unique only up to scaling, and the scale of $h_L$ is an important free parameter (see Lemma \ref{l:fix the constant A}).

Thanks to this identification, an $\ALH^*$ gravitational instanton is the same as  a complete hyperk\"ahler 4-manifold which is asymptotically Calabi in the sense of \cite[Definition 4.1]{HSVZ}.
 The proof of part (i) of Theorem~\ref{t:main2} follows from this equivalence and Theorem~\ref{t:main1}.

\subsection{Compactification to rational elliptic surfaces}
In this subsection, we prove part~(ii) of Theorem~\ref{t:main2}. 
First, we recall some basic facts about the lowest nontrivial eigenvalue of the Laplacian on $(\mathbb{T}^2,g_{c,\tau})$. The $\ZZ^2$-action on $\CC$ is generated by 
\begin{align}
(x,y) \mapsto (x + c, y), \quad (x,y) \mapsto (x + c\tau_1, y+ c\tau_2). 
\end{align}
The eigenfunctions of the Laplacian on $(\mathbb{T}^2,g_{c,\tau})$ are given by
\begin{align}
\phi_{m,n}(x,y) = e^{2 \pi i m c^{-1} \big(x - \tau_1\tau_2^{-1} y \big)} e^{2\pi i n c^{-1}\tau_2^{-1}y}
\end{align}
for $(m,n)\in\ZZ^2$. The eigenvalues are 
\begin{align}
\lambda_{m,n} = 4\pi^2 c^{-2}\{   m^2 + \tau_2^{-2} (n - m \tau_1)^2 \}.
\end{align}
Using the PSL$(2,\ZZ)$-action on $\mathbb{H}$, we can assume without loss of generality that $|\tau| \geq 1$. Then the lowest nontrivial eigenvalue is given by
$\lambda_1 =  \lambda_{0,1} = 4 \pi^2 c^{-2} \tau_2^{-2}$, with eigenfunction $\phi_{0,1}$. 

Next, we establish an almost optimal decay rate for any $\ALH^*$ gravitational instanton.

\begin{proposition}
\label{t:ALHstarorder}
Let $(X,g,\bm{\omega})$ be an $\ALH^*$ gravitational instanton with parameters $b \in \ZZ_+$, $c >0$, $\tau \in \HH$. Then, in suitable $\ALH^*$ coordinates, \eqref{metricdecay}--\eqref{omegadecay} hold with $\delta = \sqrt{\lambda_1}- \epsilon$ for any $\epsilon > 0$.
\end{proposition}

\begin{proof}
By Theorem \ref{t:main1}, $(X,g,I)$ arises as a Tian-Yau metric on a weak del Pezzo surface $M$ minus an anticanonical divisor $D$. Thus, we can use the coordinate system $\Phi$ from \cite[Proposition~3.4]{HSVZ}. 
The Tian-Yau metric is of the form $\omega = \omega_0 + \i \p \op u$ with 
$|\nabla_{g_{\mathcal{C}}}^k u|_{g_{\mathcal{C}}} = O(e^{-\epsilon z})$ as $z \to \infty$ for some $\epsilon > 0$ and all $k \geq 0$. By \cite[Proposition~3.4]{HSVZ}, the background K\"ahler form satisfies the asymptotics
\begin{align}
|\nabla^k_{g_{\mathcal{C}}}(\omega_0 - \omega_{\mathcal{C}})|_{g_{\mathcal{C}}} = O(e^{-\epsilon z^2})
\end{align}
as $z \to \infty$, for some $\epsilon > 0$ and all $k \geq 0$. 
Expanding the Calabi-Yau equation 
\begin{align}
\label{eq:CY}
( \omega_0 + \i \p \op u)^2 = \frac{1}{2} \Omega \wedge \overline{\Omega}
\end{align}
yields that $\Delta_{g_{\mathcal{C}}} u = O(e^{-2\epsilon z})$ as $z \to \infty$. Assume that $2\epsilon < \sqrt{\lambda_1}$. Then the arguments of \cite[Section~4]{HSVZ}
allow us to conclude that $u=O(e^{-2\underline{\epsilon}z})$ for any $\underline{\epsilon}\in(0,\epsilon)$. Indeed, notice that \cite[Proposition~4.12]{HSVZ} holds with $\underline{\delta} = \sqrt{\lambda_1}$, as is clear from the proof. Using this, we can find a function $u_{-2\underline{\epsilon}}$ defined on $\{z > R_1\}$ for some $R_1 > R$ such that $u_{-2\underline{\epsilon}} = O(e^{-2\underline{\epsilon} z})$ for any $\underline{\epsilon}\in(0,\epsilon)$ and $\Delta_{g_\mathcal{C}} u_{-2\underline{\epsilon}} = \Delta_{g_\mathcal{C}} u$. The function $u - u_{-2\underline{\epsilon}}$ is then harmonic with respect to $g_{\mathcal{C}}$ and is also $o(1)$ as $z \to \infty$, so from \cite[Proposition 4.10]{HSVZ} (part (2) of which is easily seen to hold with $\underline{\delta}/2$ replaced by any number greater than $-\sqrt{\lambda_1} $), we have that
\begin{align}
u - u_{-2\underline{\epsilon}} = O(e^{(- \sqrt{\lambda_1} + \epsilon') z})
\end{align}
for any $\epsilon'>0$, which implies that $u =  O(e^{-2\underline{\epsilon} z})$ for any $\underline{\epsilon}\in(0,\epsilon)$. We can iterate this argument together with standard local derivative estimates for the equation \eqref{eq:CY} to get 
\begin{align}
|\nabla^k_{g_{\mathcal{C}}}u|_{g_{\mathcal{C}}} =  O(e^{(-\sqrt{\lambda_1} + \epsilon) z})
\end{align}
as $z \to \infty$, for all $k \geq 0$ and $\epsilon>0$. This implies \eqref{omegadecay} for $i =1$ with $\delta = \sqrt{\lambda_1} - \epsilon$ for any $\epsilon > 0$.

Also, from Theorem~\ref{t:main1} and 
\cite[Proposition 3.4(a)]{HSVZ}, we have that 
\begin{align}
\label{Idecay}
|\nabla^k_{g_{\fM}}( \Phi^*I - I_{\fM})|_{g_{\fM}} = O(e^{-\epsilon' z^2})
\end{align}
for some $\epsilon' > 0$ and any $k \geq 0$ as $z \to \infty$.  Since $g$ is K\"ahler with respect to $I$, the metric condition \eqref{metricdecay} with  $\delta = \sqrt{\lambda_1} - \epsilon$  then follows from \eqref{Idecay} and the above estimates on $\Phi^* \omega_1 - \omega_1^{\fM}$.

Finally, from \cite[Proposition 3.4(b)]{HSVZ}, we have that 
\begin{align}
| \nabla^k_{g_{\fM}} (  \Phi^*\Omega_X - \Omega_{\fM})|_{g_{\fM}} = O( e^{-\epsilon' z^2})
\end{align}
for some $\epsilon' > 0$ and any $k \geq 0$ as $z \to \infty$.
Because $\Omega_X = \omega_2 + \i \omega_3$ and $\Omega_{\fM} =  \omega_2^{\fM} + \i \omega_3^{\fM}$, the conditions \eqref{omegadecay} for $i = 2$ and $i = 3$ actually hold for any  $\delta > 0$.
\end{proof}

\begin{remark}
Proposition~\ref{t:ALHstarorder} could also be proved by following the arguments in \cite[Section~6.5]{SZ21}. This way of reasoning is analytically more involved but it does not require Theorem~\ref{t:main1}. 
\end{remark}

We now produce a $J$-holomorphic function on $X$ asymptotic to a nice $J_{\fM}$-holomorphic function on $\fM$ that defines an elliptic fibration of $\fM$ with the desired behavior at infinity.

\begin{proposition} 
\label{p:holo}
Let $(X,g,I,J,K)$ be an $\ALH^*$ gravitational instanton. 
Then there exists a function $u: X \rightarrow \CC$ 
which is holomorphic with respect to $J$ and satisfies
\begin{align}
\label{uasymp}
u =  e^{\sqrt{\lambda_1}(z + iy)} + f
\end{align}
where $ |\nabla^k f|_g = O( e^{\epsilon z})$ for all $k \geq 0$,
as $z \to \infty$, for any $\epsilon >0$.  
\end{proposition}

\begin{proof} 
The function $ z + \i y$ is a locally defined $J_{\fM}$-holomorphic function on $\fM$ because 
\begin{align}
J_{\fM} (dz + i dy) = -dy + i dz = i(dz + i dy).
\end{align}
Consequently, $h = e^{\sqrt{\lambda_1}(z + iy)}$ is a globally defined $J_{\fM}$-holomorphic function on $\fM$. Identify $X$ and $\mathfrak{M}$ at infinity using the $\ALH ^*$ coordinate system of Proposition \ref{t:ALHstarorder}. Fix a cutoff function $\chi$ such that $\chi = 1$ for $z > 4 R$ and $\chi = 0$ for $z < 2R$. Using the fact that $\Delta_{g_{\fM}}h = 0$ and Proposition \ref{t:ALHstarorder}, we have that
\begin{align}
\label{deltachih}
\Delta_g (\chi h) = O( e^{\epsilon z})\;\,\text{as}\;\,z \to \infty,\;\text{for all}\;\,\epsilon >0.
\end{align}

\noindent \emph{Claim}:~There exists a function $f \in C^\infty(X)$ such that $\Delta_g f = - \Delta_g (\chi h)$ and
\begin{align}
\label{festimate}
f = O(e^{\epsilon z})\;\,\text{as}\;\,z \to \infty,\;\text{for all}\;\,\epsilon >0.
\end{align}

\noindent \emph{Proof of Claim}:~Using \cite[Proposition 4.12]{HSVZ}, we can find a smooth function $f_{\fM}$ defined on $\fM$ such that $f_{\fM} = O(e^{\epsilon z})$ as $z \to \infty$ for any $\epsilon > 0$ and such that $\Delta_{g_{\fM}}f_{\fM} = -\Delta_g h$. Thus, aiming to set $f = \chi f_{\fM} + f'$, we can reduce to solving the equation $\Delta_g f' = h'$ for $f' = O(e^{\epsilon z})$ (for any $\epsilon > 0$), where
\begin{align}
h' \equiv - \Delta_g(\chi h) - \Delta_g(\chi f_{\fM}) = O(e^{(-\sqrt{\lambda_1}+\epsilon)z})
\end{align}
for any $\epsilon > 0$ thanks to Proposition \ref{t:ALHstarorder}. In fact, we will now show that for any $\delta>0$ there exists a $\underline{\delta}>0$ such that the equation $\Delta_g f' = h'$ with $h' = O(e^{-\delta z})$ is solvable with $f' = \alpha z + O(e^{-\underline{\delta}z})$ for some $\alpha \in \mathbb{R}$. This implies what we want. To prove this new claim, we first observe that the function $z$ is $g_{\fM}$-harmonic. Thus, by Proposition \ref{t:ALHstarorder}, $\Delta_g(\chi z) = O(e^{(-\sqrt{\lambda_1}+\epsilon)z})$ for any $\epsilon > 0$, so in particular $\Delta_g(\chi z) \in L^1(X,d{\rm Vol}_g)$. The divergence theorem now allows us to conclude that $\int_X \Delta_g(\chi z)\,d{\rm Vol}_g > 0$. Indeed, the corresponding boundary integrals in $(X,g)$ and in $(\fM,g_{\fM})$ are asymptotic to each other because $z$ grows more slowly than any exponential in $z$, and the boundary integral in $(\fM, g_{\fM})$ approaches a positive constant as $z \to \infty$ by direct computation. Thus, by replacing $h'$ by $h' -\alpha \Delta_g(\chi z)$ for some suitable $\alpha \in \RR$, we can assume without loss of generality that $h'$ has mean value zero over $X$ with respect to $d{\rm Vol}_g$. It now follows from \cite[Theorem 1.5]{Hein2} that the equation $\Delta_g f' = h'$ is solvable for some smooth and uniformly bounded $f'$ with $\int_X |df'|_g^2 \, d{\rm Vol}_g < \infty$. It then also follows from the integration argument in the proof of \cite[Proposition 2.9(ia)]{Hein} that any such $f'$ satisfies the estimate 
$|\nabla^k ( f' - C)|_g = O(e^{- \underline{\delta} z})$ for some constant $C$ and for any $k \geq 0$ as $z \to \infty$.   
After subtracting this constant, the (new) claim follows.\hfill$\Box$\medskip\

Since the function $u \equiv \chi h + f$ is harmonic, 
the $1$-form $\alpha \equiv \overline{\partial}_{J} u$ satisfies
\begin{align} 
\label{halfharm}
\overline{\partial}_{J} \alpha = 0,  \ \overline{\partial}_{J,g}^* \alpha = 0.
\end{align}
Since $g$ is K\"ahler with respect to $J$, from Proposition \ref{t:ALHstarorder}, we have that
\begin{align}
\label{Jdecay}
|J - J_{\fM}|_{g_{\fM}} = O(e^{(- \sqrt{\lambda_1} + \epsilon)z})
\end{align}
as $z \to \infty$. Using the fact that $h = O(e^{\sqrt{\lambda_1}z})$ and ${\overline\partial}_{J_{\fM}} h = 0$, we can deduce from \eqref{Jdecay} that ${\overline\partial}_J h = O(e^{\epsilon z})$ for any $\epsilon>0$ with respect to $g_{\fM}$. By \eqref{festimate} and standard local gradient estimates, the same bound holds for ${\overline\partial}_J f$. Thus, it holds for $\alpha$ as well. Since $\alpha$ has $J$-type $(0,1)$, ${\rm Re}(\alpha)$ is half-harmonic with respect to $g$ thanks to \eqref{halfharm}. By \cite[Theorem~5.1]{HSVZ}, we conclude that $\alpha \equiv 0$, which implies that $u$ is $J$-holomorphic.
\end{proof}
We now complete the proof of part (ii) of Theorem~\ref{t:main2}.
Given our work so far, the proof is similar to the proof in \cite[Section~4.7]{CCI}. Let $u$ be the holomorphic function from Proposition \ref{p:holo}. From \eqref{uasymp}, it follows that all fibers of $u$ near infinity are regular and are diffeomorphic to tori, hence all the regular fibers of $u$ are diffeomorphic to tori, 
and so $u: X \rightarrow \CC$ is an elliptic fibration. The $\ALH^*$ coordinates define a submanifold near infinity by 
\begin{align}
\Sigma_0 = \{(x,y,t,z) : x = 0, t = 0\}.
\end{align}
This is clearly a smooth section of the model elliptic fibration over $U \equiv \CC \setminus B_R(0)$ for some $R \gg 1$. It is moreover 
$J_{\fM}$-holomorphic because for all $p \in \Sigma_0$, 
\begin{align}
T_p \Sigma_0 = \mathrm{span}_p \Big\{\frac{\p}{\p y}, \frac{\p}{\p z} \Big\}
=  \{ Y \in T_{p} \fM  : e_1(Y) = e_4(Y) = 0 \},
\end{align}
and $\mathrm{span}\{ e_1, e_4 \}$ is invariant under $J_{\fM}$ by \eqref{Jstar}.

We view the section $\Sigma_0$ as a mapping $\sigma_0 : U \rightarrow X \setminus X_R$. We next want to perturb $\sigma_0$ to a section $\sigma$ which is holomorphic with respect to $J$. Given any smooth (not necessarily holomorphic) section $\sigma$ over $U$, we can define $\op \sigma \in \Gamma( \Lambda^{0,1}(U) \otimes \sigma^* T^{1,0}(F))$,  where $T^{1,0}(F)$ is the $(1,0)$ part of the vertical tangent bundle, by restricting the differential $\sigma_* \otimes \CC$ to $T^{0,1}(U)$, and then projecting to the $(1,0)$ part. 
Next, we use the $2$-form $\Omega = \omega_1 + \i \omega_3$, which is holomorphic with respect to $J$. If we insert the $T^{1,0}(F)$ component of $\op \sigma$ into $\Omega$, we can define $\Omega \odot \op \sigma \in C^{\infty}(U, \CC)$, since $\Lambda^{1,0}(U) \otimes \Lambda^{0,1}(U) \cong \Lambda^{1,1}(U)$ is a trivial bundle. 
Denoting $h(u, \bar{u}) = \Omega \odot \op \sigma_0$, from basic theory of the $\overline{\partial}$-operator in $U$, we can solve the equation $\frac{\partial}{\partial \bar{u}}H = h$ on $U$. 
Choose an arbitrary point $p \in U$ and an affine holomorphic fiber coordinate $w$ over a small neighborhood $U_p$ of $p$ in $U$. Then $\{w = 0\}$ is a local holomorphic section over $U_p$, and the holomorphic $2$-form can be written as
$\Omega = f(u) du \wedge dw$, where $f(u)$ is nowhere vanishing. It is easy to see that the smooth local sections
\begin{align}
\sigma_{h,p} \equiv \left( \frac{H(u,\bar{u})}{f(u)} \right) \frac{\p}{\p w}
\end{align}
over $U_p$ patch up to a well-defined smooth section $\sigma_h$ over $U$, independent of the choice of local $w$ coordinate. Consequently, the section $\sigma \equiv \sigma_0 - \sigma_h$ is a holomorphic section defined over all of $U$. 

After identifying $U$ with a punctured disc $\Delta^*$ using $z = u^{-1}$ as a holomorphic coordinate, we can then identify the elliptic surface with $(\Delta^*_z \times \CC_w) / (\ZZ \oplus \ZZ )$, with the action given by 
\begin{align}
\label{quot}
(m,n) \cdot (z,w) = (z, w + m t_1(z) + n t_2(z)), \ (m,n) \in \ZZ \oplus \ZZ
\end{align}
($t_1, t_2$ are the periods), such that $\{w = 0\}$ defines $\sigma$; 
see for example \cite[pp.369--370]{Hein}. 
Consequently, there exists a compactification of $X$ to an elliptic surface $S$ such that $X = S \setminus F$, where $F$ is the fiber at infinity. Since the cross-sections are diffeomorphic to nilmanifolds of degree $b$, the only possibility is that the monodromy is of type $\I_b$. It is easy to see that the form $\Omega$ is then a meromorphic $2$-form on $S$ with a pole of order $1$ along $F$, which implies that $F$ has multiplicity~$1$. Since $\mathrm{div}(\Omega) = -F$, we have that $K_S = - [F]$. 
From Corollary~\ref{c:main1}, $b^1(X) = 0$. A Mayer-Vietoris argument similar to that in the proof of Corollary~\ref{c:main1} above shows that $b^1(S) = 0$. Arguing exactly as in \cite[Theorem~3.3]{CCIII}, we see that $S$ is a rational elliptic surface, with projection $u : S \rightarrow \PP^1$, so is the blow-up of $\PP^2$ in $9$ points. Consequently, there exists a $(-1)$-curve $E$ (the exceptional divisor of the last blow-up). The adjunction formula then implies that $K_S \cdot E = - 1$, so the condition that $K_S = - [F]$ implies that there are no multiple fibers, and $E$ is a global section; see \cite[Proposition 4.1]{HarbourneLang}.

\bibliographystyle{amsplain} 
\bibliography{HSVZ2}

\end{document}